\documentclass{amsart}

\usepackage{amsmath,amsthm,amssymb,amscd}
\usepackage[usenames, dvipsnames]{color}
\usepackage{mathtools}

\newtheorem{thm}{Theorem}[section]
\newtheorem{lem}[thm]{Lemma}
\newtheorem{cor}[thm]{Corollary}
\newtheorem{problem}[thm]{Problem}

\newtheorem{prop}[thm]{Proposition}

\newtheorem{ex}[thm]{Example}
\newtheorem{rem}[thm]{Remark}

\newcommand{\R}{\mathbb{R}}
\newcommand{\N}{\mathbb{N}}
\newcommand{\I}{\mathcal{I}}

\renewcommand{\phi}{\varphi}
\renewcommand{\epsilon}{\varepsilon}
\renewcommand{\emptyset}{\varnothing}

\begin{document}

\title{Differentiability of continuous functions in terms of Haar-smallness}

\author[A. Kwela]{Adam Kwela}
\address{Institute of Mathematics, Faculty of Mathematics, Physics and Informatics, University of Gda\'{n}sk, ul.~Wita  Stwosza 57, 80-308 Gda\'{n}sk, Poland}
\email{Adam.Kwela@ug.edu.pl}

\author[W.A. Wo{\l}oszyn]{Wojciech Aleksander Wo{\l}oszyn}
\address{Institute of Mathematics, Faculty of Mathematics, Physics and Informatics, University of Gda\'{n}sk, ul.~Wita  Stwosza 57, 80-308 Gda\'{n}sk, Poland}
\email{woloszyn.w.a@gmail.com}


\begin{abstract}
One of the classical results concerning differentiability of continuous functions states that the set $\mathcal{SD}$ of somewhere differentiable functions (i.e., functions which are differentiable at some point) is Haar-null in the space $C[0,1]$. By a recent result of Banakh et al., a set is Haar-null provided that there is a Borel hull $B\supseteq A$ and a continuous map $f\colon \{0,1\}^\N\to C[0,1]$ such that $f^{-1}[B+h]$ is Lebesgue's null for all $h\in C[0,1]$. 

We prove that $\mathcal{SD}$ is not Haar-countable (i.e., does not satisfy the above property with "Lebesgue's null" replaced by "countable", or, equivalently, for each copy $C$ of $\{0,1\}^\N$ there is an $h\in C[0,1]$ such that $\mathcal{SD}\cap (C+h)$ is uncountable.

Moreover, we use the above notions in further studies of differentiability of continuous functions. Namely, we consider functions differentiable on a set of positive Lebesgue's measure and functions differentiable almost everywhere with respect to Lebesgue's measure. Furthermore, we study multidimensional case, i.e., differentiability of continuous functions defined on $[0,1]^k$. Finally, we pose an open question concerning Takagi's function.
\end{abstract}

\maketitle

\section{Introduction}

We follow the standard topological notation and terminology. By $|X|$ we denote the cardinality of a set $X$.

For a function $f\in C[0,1]$, by $D(f)$ we denote the set of all points $x\in [0,1]$ at which $f$ is differentiable. There are examples of continuous functions such that $D(f)=\emptyset$. One of the first and simplest examples is the famous Takagi's function $T\colon\R\to\R$ given by $T(x)=\sum_{i=0}^\infty \frac{1}{2^i}\text{dist}(2^i x,\mathbb{Z})$ (see \cite{Allaart}, \cite{survey}, or \cite{Takagi}). The size of the set of somewhere differentiable functions, i.e., functions $f$ such that $D(f)\neq\emptyset$, is a classical object of studies since Banach's result stating that this set is meager in $C[0,1]$ (cf. \cite{Banach}). One of the well-known results in this subject is Hunt's theorem stating that the aforementioned set is Haar-null in the space $C[0,1]$ (see \cite{Hunt}). This notion was first introduced in \cite{Christensen}  by Christensen. He called a subset $A$ of an abelian Polish group $X$ \emph{Haar-null} provided that there is a Borel hull $B\supseteq A$ and a Borel $\sigma$-additive probability measure $\lambda$ on $X$ such that $\lambda(B+x)=0$ for all $x\in X$ (actually, in the original paper Christensen introduced this notion using universally measurable hulls instead of Borel hulls, however in many papers, e.g. \cite{1} or \cite{Darji}, Haar-null sets are defined in the same way as above). A big advantage of this concept is that in a locally compact group it is equivalent to the notion of Haar measure zero sets and at the same time it can be used in a significantly larger class of groups. In \cite{HSY} Hunt, Sauer and Yorke, unaware of Christensen's paper, reintroduced the notion of Haar-null sets in the context of dynamical systems (in their paper Haar-null sets are called shy sets, and their complements are called prevalent sets). 

Actually, since the set of somewhere differentiable functions is not Borel, Hunt had to show something more: the set of somewhere Lipschitz functions is Haar-null in $C[0,1]$ (a function $f\in C[0,1]$ is \emph{somewhere Lipschitz} whenever there is an $x\in[0,1]$ and an $M\in\mathbb{N}$ such that $|f(x)-f(y)|\leq M|x-y|$ for each $y\in[0,1]$; observe that each somewhere differentiable function is somewhere Lipschitz).

In this paper, we are interested in the following notions of smallness. A subset $A$ of an abelian Polish group $X$ is:
\begin{itemize} 
	\item Haar-countable if there is a Borel hull $B\supseteq A$ and a copy $C$ of $\{0,1\}^\N$ such that $(C+x)\cap B$ is countable for all $x\in X$;
	\item Haar-finite if there is a Borel hull $B\supseteq A$ and a copy $C$ of $\{0,1\}^\N$ such that $(C+x)\cap B$ is finite for all $x\in X$;
	\item Haar-$n$, for $n\in\mathbb{N}$, if there is a Borel hull $B\supseteq A$ and a copy $C$ of $\{0,1\}^\N$ such that $|(C+x)\cap B|\leq n$ for all $x\in X$.
\end{itemize}
Clearly,
$$\text{Haar-}n\ \Longrightarrow\ \text{Haar-}(n+1)\ \Longrightarrow\ \text{Haar-finite}\ \Longrightarrow\ \text{Haar-countable}$$
for any $n\in\N$. 

The choice of names in the above is due to Banakh et al., who recently unified the notions of Haar-null sets and Haar-meager sets (defined by Darji in \cite{Darji}) in \cite{1} by introducing the concept of Haar-small sets. A collection of subsets of a set $X$ is called a semi-ideal whenever it is closed under taking subsets. Following \cite[Definition 11.1]{1}, for a semi-ideal $\I$ on the Cantor cube $\{0,1\}^\N$, we say that a subset $A$ of an abelian Polish group $X$ is Haar-$\I$ if there is a Borel hull $B\supseteq A$ and a continuous map $f\colon \{0,1\}^\N\to X$ such that $f^{-1}[B+x]\in\I$ for all $x\in X$. It turns out that if $\I$ is the $\sigma$-ideal $\mathcal{N}$ of subsets of $\{0,1\}^\N$ of Haar measure zero, then we obtain Haar-null sets (cf. \cite[Theorem 11.5]{1}). The same holds for the $\sigma$-ideal $\mathcal{M}$ of meager sets and Haar-meager sets (cf. \cite[Theorem 11.5]{1}), the $\sigma$-ideal of countable sets and Haar-countable sets, the ideal of finite sets and Haar-finite sets and the semi-ideal of sets of cardinality at most $n$ and Haar-$n$ sets (cf. \cite[Proposition 1.2]{ja}).

Obviously, the collection of Haar-$\I$ subsets of an abelian Polish group is a semi-ideal. Observe that $\I\subseteq\mathcal{J}$ implies that each Haar-$\I$ set is Haar-$\mathcal{J}$. Thus, Haar-null sets and Haar-meager sets only allow us to say that some properties are rare (and we cannot put them on a scale and compare with other rare properties), whereas Haar-$\I$ sets allow us to develop a whole hierarchy of small sets.

Haar-countable, Haar-finite, and Haar-$n$ sets were profoundly studied by the first author in \cite{ja}. There are compact examples showing that none of the above implications can be reversed in $C(K)$ where $K$ is compact metrizable. Moreover, it is known that neither Haar-finite sets nor Haar-$n$ sets form an ideal (see \cite[Corollary 5.2 and Theorem 6.1]{ja}). Zakrzewski considered Haar-small sets in \cite{Zakrzewski} under the name of perfectly $\kappa$-small sets. A particular case of Haar-$1$ sets was investigated by Balcerzak in \cite{Balcerzak}. He introduced the so-called property (D) of a $\sigma$-ideal $\I$, which says that there is a Borel Haar-$1$ set not belonging to $\I$. Moreover, Banakh, Lyaskovska, and Repov\v s considered packing index of a set in \cite{BLR}. Packing index is closely connected to Haar-$1$ sets (namely, a Borel set is Haar-$1$ if and only if its packing index is uncountable). In turn, Haar-countable sets were studied by Darji and Keleti in \cite{DK} and by Elekes and Stepr\=ans in \cite{Elekes}.

\section{Nowhere differentiable functions}

Hunt proved in \cite{Hunt} that the set $\mathcal{SD}$ of somewhere differentiable functions is Haar-null in $C[0,1]$. Actually, a closer look at his proof gives us something more. Denote by $\mathcal{E}$ ($\mathcal{E}(\mathbb{R})$) the $\sigma$-ideal generated by compact Haar null subsets of $\{0,1\}^\mathbb{N}$ ($\mathbb{R}$, respectively) and recall that $\mathcal{E}$ is a proper subfamily of $\mathcal{N}\cap\mathcal{M}$ (see \cite[Lemma 2.6.1]{Judah}). Hunt's proof shows that there is a continuous function $\phi:\mathbb{R}^2\to C[0,1]$ such that $\phi^{-1}[f+\mathcal{SD}]\in\mathcal{E}(\mathbb{R})$ for all $f\in C[0,1]$. If $\psi:\{0,1\}^\mathbb{N}\to[0,1]^2$ denotes the map given by $\psi((x_i))=(\sum_{i\in\mathbb{N}}\frac{x_{2i-1}}{2^i},\sum_{i\in\mathbb{N}}\frac{x_{2i}}{2^i})$ then $\phi^{-1}[A]\in\mathcal{E}$ for all $A\in\mathcal{E}(\mathbb{R})$. Thus, $\psi^{-1}[\phi^{-1}[f+\mathcal{SD}]]\in\mathcal{E}$ for all $f\in C[0,1]$, i.e., we have shown the following:

\begin{thm}[{Hunt, \cite{Hunt}}]
The set $\mathcal{SD}$ is Haar-$\mathcal{E}$ in $C[0,1]$.
\end{thm}

In this Section we show that $\mathcal{SD}$ is not Haar-countable in $C[0,1]$.

\begin{thm}
\label{tw1}
The set of somewhere differentiable functions is not Haar-countable in $C[0,1]$.
\end{thm}

\begin{proof}
Denote by $C$ the ternary Cantor set (which is homeomorphic to $\{0,1\}^\N$). Let $\phi\colon C\to C[0,1]$ be continuous. We need to find a homeomorphic copy $C'$ of $C$ and a continuous function $g\colon[0,1]\to\R$ such that $\phi(c)-g\in\mathcal{SD}$ for all $c\in C'$.

Define $D_1=\{0,1\}$ and $D_n=\left(\left\{\frac{i}{3^{n-1}}:\ i\in\N\right\}\cap C\right)$ for all $n\in\N$, $n\neq 1$. Let $E_1=D_1$ and $E_n=D_n\setminus D_{n-1}$, for $n\in\N$, $n>1$. Observe that $D=\bigcup_{n\in\omega}D_n$ is the set of all points from $C$ with finite ternary expansion. For each $n\in\N$ and $d\in E_n$, let: 
$$U_d=[0,1]\cap \left[d-\frac{\frac{1}{3^{n-1}}+\frac{1}{3^{n+1}}}{2},d+\frac{\frac{1}{3^{n-1}}+\frac{1}{3^{n+1}}}{2}\right].$$
Note that $U_d\cap D_n=\{d\}$, for every $d\in D_n$. 

First, we need to shrink $C$. Namely, we want to have a homeomorphic copy $C'\subseteq C$ of $C$ such that:
\begin{itemize} 
 \item $\left\|f_c-f_{c'}\right\|\leq (c-c')^2$, for all $c,c'\in C$;
 \item $\left\|f_d-f_{d'}\right\|<\inf_{x\in U_d\cap U_{d'}}\left((d-x)^2+(d'-x)^2\right)$, for all distinct $d,d'\in D$ with $U_d\cap U_{d'}\neq\emptyset$;
 \item $\left\|f_d-f_{d'}\right\|<\frac{1}{5}\left(\text{dist}\left(d,U_{d'}\right)\right)^2=\frac{1}{5}\left(\text{dist}\left(d',U_{d}\right)\right)^2$, whenever $d,d'\in E_n$ for some $n$, $d\neq d'$ and $U_d\cap U_{d'}\neq\emptyset$;
 \item $\left\|f_d-f_{d'}\right\|<\frac{1}{5}\left(\text{dist}\left(d,U_{d'}\right)\right)^2$, whenever $d'\in E_n$ and $d\in D_n\setminus E_n$ for some $n$ and $U_d\cap U_{d'}\neq\emptyset$;
\end{itemize}
where $f_x=\phi(\psi(x))$ and $\psi\colon C\to C'$ is the aforementioned homeomorphism. 

The construction of $C'$ is rather standard. Nevertheless, we provide a short sketch of it. 

Let $e\colon C^2\to\R$ be a continuous function such that $\left\|f_c-f_{c'}\right\|\leq e(c,c')$ guarantees all of the above conditions and $e(c,c')>0$ for all $c,c'\in C$ with $c\neq c'$ (which clearly exists). For a finite sequence $s=(s_1,\ldots,s_{n})\in\{0,1\}^{n}$, let $\overline{s}=\sum_{i=1}^{n}\frac{2s_i}{3^i}$. In the first inductive step pick any basic clopen set $W_{\emptyset}\subseteq C$ such that:
$$\text{diam}\left(\phi[W_{\emptyset}]\right)<\inf\left\{e(c,c'):\ c\in C\cap\left[0,\frac{1}{3}\right]\ \wedge\ c'\in C\cap\left[\frac{2}{3},1\right]\right\}.$$
If $W_s$ are already defined for all $s\in\{0,1\}^n$, for each $s=(s_1,\ldots,s_{n})\in\{0,1\}^{n}$ find two disjoint basic clopen sets $W_{s^\frown(0)},W_{s^\frown(1)}\subseteq W_s$ such that $\text{diam}\left(\phi[W_{s^\frown(i)}]\right)<\frac{1}{2^{n+1}}$ and:
$$\text{diam}\left(\phi[W_{s^\frown(i)}]\right)<\inf\left\{e(c,c'):\ c\in C\cap\left[\overline{s^\frown(i,0)},\overline{s^\frown(i,0)}+\frac{1}{3^{n+1}}\right]\ \wedge\right.$$
$$\left.\wedge\ c'\in C\cap\left[\overline{s^\frown(i,1)},\overline{s^\frown(i,1)}+\frac{1}{3^{n+1}}\right]\right\},$$
for $i=0,1$. It is easy to check that $C'=\bigcup_{x\in\{0,1\}^\N}\bigcap_{n\in\N}W_{x\upharpoonright\{1,\ldots,n\}}$ is the required set.

Now, we want to construct a $g\in C[0,1]$ such that $f_c-g$ has a derivative at $c$ equal to $0$, for all $c\in C$. Inductively, we will define a sequence of continuous functions $(g_n)\subseteq C[0,1]$ such that:
$$\forall_{d\in E_n}\ \forall_{x\in U_d}\ \left|f_d(x)-\sum_{i=1}^n g_i(x)\right|\leq\frac{1}{5}(x-d)^2;$$
$$\forall_{d\in D_n}\ \forall_{x\in U_d}\ \left|f_d(x)-\sum_{i=1}^n g_i(x)\right|\leq (x-d)^2.$$
At the end, we will put $g=\sum_{i=1}^\infty g_n$.

Start the construction with $g_1\in C[0,1]$ such that:
\begin{itemize}
    \item $g_1=f_0$ on $U_0\setminus U_1$;
    \item $g_1=f_1$ on $U_1\setminus U_0$;
    \item $g_1(x)$ is between $f_0(x)$ and $f_1(x)$ for all $x\in U_0\cap U_1$.
\end{itemize} 
Note that we have $\left|f_0(x)-g_1(x)\right|=0\leq\frac{1}{5}(x-0)^2$ for all $x\in U_1\setminus U_0$ and $\left|f_0(x)-g_1(x)\right|\leq\left\|f_0-f_1\right\|<\frac{1}{5}(\inf U_1-0)^2\leq\frac{1}{5}(x-0)^2$ for all $x\in U_0\cap U_1$. Similarly, $\left|f_1(x)-g_1(x)\right|\leq\frac{1}{5}(x-1)^2$ for all $x\in U_1$. Thus, $g_1$ is as needed.

Once all $g_i$'s, for $i<n$, are defined, let $\tilde{g}_n\colon\bigcup_{d\in E_n}U_d\to\mathbb{R}$ be a continuous function such that:
\begin{itemize}
    \item $\tilde{g}_{n}=f_d-\sum_{i=1}^{n-1}g_i$ on $U_d\setminus \bigcup_{d'\in E_n\setminus\{d\}}U_{d'}$ for each $d\in E_n$;
    \item if $d,d'\in E_n$ and $U_d\cap U_{d'}\neq\emptyset$, then $\tilde{g}_n(x)$ is between $f_d(x)-\sum_{i=1}^{n-1}g_i(x)$ and $f_{d'}(x)-\sum_{i=1}^{n-1}g_i(x)$ for all $x\in U_d\cap U_{d'}$.
\end{itemize}
Let $d\in E_n$ and notice that $\left|f_d(x)-\sum_{i=1}^{n-1} g_i(x)-\tilde{g}_n(x)\right|=0\leq\frac{1}{5}(x-d)^2$ for all $x\in U_d\setminus \bigcup_{d'\in E_n\setminus\{d\}}U_{d'}$ and $\left|f_d(x)-\sum_{i=1}^{n-1} g_i(x)-\tilde{g}_n(x)\right|\leq\left\|f_d-f_{d'}\right\|<\frac{1}{5}(\text{dist}(d,U_{d'}))^2\leq\frac{1}{5}(x-d)^2$ for all $x\in U_d\cap U_{d'}$ whenever the latter intersection is non-empty. Moreover, 
if $d\in D_n\setminus E_n$ and $x\in U_d\cap U_{d'}$ for some $d'\in E_n$, then:
$$|d'-x|\leq \frac{\frac{1}{3^{n-1}}+\frac{1}{3^{n+1}}}{2}\leq \frac{1}{3^{n-1}}-\frac{1}{3^{n+1}}=$$
$$2\left(\frac{1}{3^{n-1}}-\frac{\frac{1}{3^{n-1}}+\frac{1}{3^{n+1}}}{2}\right)\leq 2\text{dist}\left(d,U_{d'}\right)\leq 2|d-x|.$$
Therefore,
$$\left|f_d(x)-\sum_{i=1}^{n-1} g_i(x)-\tilde{g}_n(x)\right|\leq\left\|f_d-f_{d'}\right\|+\left|f_{d'}(x)-\sum_{i=1}^{n-1} g_i(x)-\tilde{g}_n(x)\right|\leq$$
$$\leq\frac{1}{5}\left(\text{dist}\left(d,U_{d'}\right)\right)^2+\frac{1}{5}(d'-x)^2\leq\frac{1}{5}(d-x)^2+\frac{4}{5}(d-x)^2=(d-x)^2.$$
Hence, $\tilde{g}_n$ satisfies all the required conditions and it suffices to extend it to a continuous function $g_n$ defined on the whole interval $[0,1]$ such that:
$$\left(f_d(x)-\sum_{i=1}^{n-1}g_i(x)-(d-x)^2\right)\leq g_n(x)\leq\left(f_d(x)-\sum_{i=1}^{n-1}g_i(x)+(d-x)^2\right),$$
for all $d\in D_n$ such that $x\in U_d$. Note that such extension exists. Indeed, we need to check that whenever $x\in U_d\cap U_{d'}$ for some $d,d'\in D_n\setminus E_n$ then 
$$f_d(x)-\sum_{i=1}^{n-1}g_i(x)-(d-x)^2\leq f_{d'}(x)-\sum_{i=1}^{n-1}g_i(x)+(d'-x)^2.$$
The above leads to $|f_d(x)-f_{d'}(x)|\leq (d-x)^2+(d'-x)^2$ which is satisfied as $\left\|f_d-f_{d'}\right\|<\inf_{x\in U_d\cap U_{d'}}\left((d-x)^2+(d'-x)^2\right)$, for all distinct $d,d'\in D$ with $U_d\cap U_{d'}\neq\emptyset$.

Once the construction is completed, define $g=\sum_{n=1}^\infty g_n$. For each $d\in D$ and $x\in U_d$, we have $|f_d(x)-g(x)|\leq (x-d)^2$. Thus, $f_d-g$ has a derivative at $d$ equal to $0$. Now, we want to show that $f_c-g$ has a derivative at $c$ equal to $0$ for each $c\in C$.

Fix $c\in C\setminus D$ and $h\in\R$ such that $c+h\in[0,1]$. Assume that $c+h\in U_d$ for some $d\in D$ and $|h|\geq\frac{1}{4}|c-d|$. Observe that:
$$\left|\frac{f_c(c+h)-g(c+h)-f_c(c)+g(c)}{h}\right|\leq$$
$$\leq\left|\frac{f_c(c+h)-f_d(c+h)}{h}\right|+\left|\frac{f_d(c+h)-g(c+h)-f_d(c)+g(c)}{h}\right|+\left|\frac{f_d(c)-f_c(c)}{h}\right|\leq $$
$$\leq 2\frac{\left\|f_c-f_d\right\|}{|h|}+\left|\frac{f_d(c+h)-g(c+h)-f_d(c)+g(c)}{h}\right|.$$
As $\left\|f_c-f_d\right\|\leq (c-d)^2$ and $|h|\geq\frac{1}{4}|c-d|$, we get that $\frac{\left\|f_c-f_d\right\|}{|h|}\leq 4|c-d|$. What is more, if $d'\in D$ is such that $c\in U_{d'}$ then 
$$|f_d(c)-g(c)|\leq|f_d(c)-f_{d'}(c)|+|f_{d'}(c)-g(c)|\leq(d-d')^2+(d'-c)^2\leq(d-c)^2.$$
Using $|h|\geq\frac{1}{4}|c-d|$ once again, we have:
$$\left|\frac{f_d(c+h)-g(c+h)-f_d(c)+g(c)}{h}\right|\leq\frac{(c+h-d)^2+(c-d)^2}{|h|}=$$
$$=\frac{2(c-d)^2+2h(c-d)+h^2}{|h|}\leq 8|c-d|+2|c-d|+|h|.$$
Thus, $\frac{f_c(c+h)-g(c+h)-f_c(c)+g(c)}{h}$ tends to $0$ as $h\to 0$ and $|c-d|\to 0$. Therefore, to finish the proof, it suffices to show that for each $c\in C$ there are sequences $(d_n)\subseteq D\cap (c,1]$ and $(d'_n)\subseteq D\cap [0,c)$ converging to $c$ and such that $(c,1]\subseteq\bigcup_{n\in\N}U_{d_n}\cap [c+\frac{1}{4}(d_n-c),1]$ and $[0,c)\subseteq\bigcup_{n\in\N}U_{d'_n}\cap [0,c-\frac{1}{4}(c-d'_n)]$.

We will construct the sequence $(d_n)$. Construction of $(d'_n)$ is similar. Let $(c_n)\in\{0,2\}^\N$ be the ternary expansion of $c\in C$ and $(i_n)\subseteq\N$ be the increasing enumeration of the set $\{n\in\N:\ c_n=0\}$. For all $k\in \N$, define $d_1=1$, $d_{2k}=1-\sum_{j=1}^{k-1}\frac{2}{3^{i_j}}-\frac{1}{3^{i_k}}$, and $d_{2k+1}=1-\sum_{j=1}^{k-1}\frac{2}{3^{i_j}}-\frac{2}{3^{i_k}}$. Now, we show that $(d_n)$ is as required. 

We need to show that $\sup U_{d_{n+1}}>c+\frac{1}{4}(d_{n}-c)$ for all $n\in\N$. Clearly, $\sup U_{d_{2k}}\geq d_{2k}+\frac{\frac{1}{3^{i_k}}+\frac{1}{3^{i_k+2}}}{2}$ and $\sup U_{d_{2k+1}}\geq d_{2k+1}+\frac{\frac{1}{3^{i_k}}+\frac{1}{3^{i_k+2}}}{2}$, i.e., it suffices to show: 
$$d_{2k}+\frac{\frac{1}{3^{i_k}}+\frac{1}{3^{i_k+2}}}{2}>c+\frac{1}{4}(d_{2k-1}-c);$$
$$d_{2k+1}+\frac{\frac{1}{3^{i_k}}+\frac{1}{3^{i_k+2}}}{2}>c+\frac{1}{4}(d_{2k}-c).$$
As $d_{2k}=c+(d_{2k}-c)$ and $d_{2k+1}=c+(d_{2k+1}-c)$, this is equivalent to:
$$d_{2k}-c+\frac{\frac{1}{2}+\frac{1}{18}}{3^{i_k}}>\frac{1}{4}(d_{2k-1}-c);$$
$$d_{2k+1}-c+\frac{\frac{1}{2}+\frac{1}{18}}{3^{i_k}}>\frac{1}{4}(d_{2k}-c).$$
Multiplying both sides by $4$ and using $d_{2k-1}=d_{2k}+\frac{1}{3^{i_k}}$ and $d_{2k}=d_{2k+1}+\frac{1}{3^{i_k}}$ we get:
$$4(d_{2k}-c)+\frac{2+\frac{2}{9}}{3^{i_k}}>d_{2k}-c+\frac{1}{3^{i_k}};$$
$$4(d_{2k+1}-c)+\frac{2+\frac{2}{9}}{3^{i_k}}>d_{2k+1}-c+\frac{1}{3^{i_k}}.$$
Thus,
$$3(d_{2k}-c)>\frac{1-2-\frac{2}{9}}{3^{i_k}};$$
$$3(d_{2k+1}-c)>\frac{1-2-\frac{2}{9}}{3^{i_k}};$$
which is true since $d_{2k}>c$ and $d_{2k+1}>c$. This finishes the entire proof.
\end{proof}

By the above, the set of functions differentiable at some point is not Haar-countable. However, what about functions differentiable at more than one point? As for a given $\sigma$-ideal $\I$ on $[0,1]$ the set $\{f\in C[0,1]:\ \emptyset\neq D(f)\in\I\}$ is contained in the set of somewhere differentiable functions, this question is natural. The following slight strengthening of Theorem \ref{tw1} gives only a partial answer to this problem. 

\begin{cor}
Let $\I$ be a $\sigma$-ideal on $[0,1]$ containing some perfect set. The set of functions $f\in C[0,1]$ such that $D(f)$ is a nonempty set which belongs to $\I$ is not Haar-countable.
\end{cor}

\begin{proof}
First, assume that the ternary Cantor set $C$ belongs to $\I$. We need to make two modifications of the proof of Theorem \ref{tw1}.

Since, by Hunt's result, the set of somewhere Lipschitz functions is Borel and Haar-null (see \cite{Hunt}), the set $\mathcal{NL}$ of nowhere Lipschitz functions cannot be Haar-null. Thus, for $\phi$ from the proof of Theorem \ref{tw1}, there is a $z\in C[0,1]$ such that $\phi^{-1}[\mathcal{NL}-z]$ is not Lebesgue's null. In particular, this is a Borel uncountable subset of $C$. Hence, it must contain a homeomorphic copy $P$ of $C$. Then, $\phi(c)+z$ is nowhere differentiable for each $c\in P$. Thus, by performing the construction of $C'$ inside $P$ and defining $f_x=\phi(\psi(x))+z$ where $\psi$ is a homeomorphism from $C$ to $C'$, we may assume that $f_c$ is nowhere differentiable for each $c\in C$. Moreover, this changes do not affect the rest of the proof. If we find $g\in C[0,1]$ such that $f_c-g$ is somewhere differentiable for all $c\in C$, then $\phi(c)+(z-g)$ is somewhere differentiable for uncountably many $c\in C$ (namely, for all $c\in C'$).

Now, we move to the second modification. We can ensure that $g$ is differentiable outside $C$. Indeed, for each connected component $I$ of $[0,1]\setminus C$ fix a sequence of closed (possibly empty) intervals $(I_n)$ such that $\bigcup_n I_n=I$, $I_n\subseteq I_{n+1}$ and $I_n\cap \bigcup_{d\in E_n} U_d=\emptyset$ for all $n$ (which is possible as $\bigcup_n E_n\subseteq C$ and $\sup_{d\in E_n}\text{diam}(U_d)$ tends to $0$ as $n\to+\infty$). Now it suffices to require, additionally, in the inductive construction of $(g_n)$, that $\sum_{i=1}^n g_i$ is differentiable on $I_n$ and $g_n\upharpoonright I_{n-1}=0$ (if $n>1$). This can be done as for $x\in I_n$ the only requirement imposed on $g_n$ resulting from the proof of Theorem \ref{tw1} is:
$$\sup_{d\in D_n, x\in U_d}\left(f_d(x)-(d-x)^2\right)\leq\sum_{i=1}^{n}g_i(x)\leq\inf_{d\in D_n, x\in U_d}\left(f_d(x)+(d-x)^2\right).$$
Thus, there is no problem with the request $g_n\upharpoonright I_{n-1}=0$ as $\left(\bigcup_{d\in D_{n}}U_d\right)\cap I_{n-1}\subseteq\left(\bigcup_{d\in D_{n}}U_d\right)\cap I_{n}\subseteq\bigcup_{d\in D_{n-1}}U_d$ and in the $n$th induction step we already have:
$$\sup_{d\in D_{n-1}, x\in U_d}\left(f_d(x)-(d-x)^2\right)\leq\sum_{i=1}^{n-1}g_i(x)\leq\inf_{d\in D_{n-1}, x\in U_d}\left(f_d(x)+(d-x)^2\right).$$
What is more, $\sum_{i=1}^{n-1}g_i(x)$ is differentiable on $I_{n-1}$, hence so is $\sum_{i=1}^{n}g_i(x)$. Since 
$$\sup_{d\in D_n, x\in U_d}\left(f_d(x)-(d-x)^2\right)<\inf_{d\in D_n, x\in U_d}\left(f_d(x)+(d-x)^2\right)$$
(by the fact that $\left\|f_d-f_{d'}\right\|<\inf_{x\in U_d\cap U_{d'}}\left((d-x)^2+(d'-x)^2\right)$, for all distinct $d,d'\in D$ with $U_d\cap U_{d'}\neq\emptyset$), there is some freedom in the choice of $g_n$ and we can ensure that $\sum_{i=1}^n g_i$ is differentiable also on $I_n\setminus I_{n-1}$.

After these modifications, as $g$ is differentiable at each point $x\in[0,1]\setminus C$ while $f_c$ is not, $f_c-g$ is not differentiable at each point of $[0,1]\setminus C$ and we get that $D(f_c-g)\subseteq C$ for all $c\in C$.

The case where $C\notin\I$ requires one additional modification. Since every perfect set contains a homeomorphic copy of the ternary Cantor set $C$, we simply need to find such a copy $R$ that belongs to $\I$. Then, we can replace $C$ with $R$, modify sets $D_n$ and $U_d$, and perform similar reasoning as in the proof of Theorem \ref{tw1}. 
\end{proof}

It is known that $D(f)$ is Borel (of type $\mathtt{G_{\delta\sigma}}$) for each $f\in C[0,1]$ (see \cite{Zahorski}). Thus, we can consider Lebesgue's measure of the set $D(f)$. Since there are perfect sets of Lebesgue's measure zero, the following is immediate.

\begin{cor}
The set of functions $f\in C[0,1]$ such that $D(f)$ is a nonempty set of Lebesgue's measure zero is not Haar-countable.
\end{cor}

As the $\sigma$-ideal of countable sets does not contain any perfect set, the following question arises.

\begin{problem}
Is the set of functions $f\in C[0,1]$ such that $D(f)$ is countable but non-empty Haar-countable in $C[0,1]$? 
\end{problem}

\section{Differentiability and Lebesgue's measure}

In this Section, we examine functions differentiable on a set of positive Lebesgue's measure. 

We will need the following notation. By the symbol $\lambda$ we will denote the Lebesgue's measure on $[0,1]$. Moreover, for a function $f\in C[0,1]$ and $M\in\mathbb{N}$, define 
$$L_M(f)=\left\{x\in[0,1]:\ \forall_{y\in[0,1]}\ |f(x)-f(y)|\leq M|x-y|\right\}.$$
Then, $f$ is somewhere Lipschitz if and only if the set $L(f)=\bigcup_{M\in\mathbb{N}}L_M(f)$ is non-empty. 

The next two rather folklore lemmas will be useful in our further considerations.

\begin{lem}
\label{lem3}
For any $f\in C[0,1]$ and $M\in\mathbb{N}$, the set $L_M(f)$ is closed. 
\end{lem}

\begin{proof}
We will show that $[0,1]\setminus L_M(f)$ is open. Fix any $x\in[0,1]\setminus L_M(f)$. Then, there is $y\in[0,1]$ such that $|f(x)-f(y)|>M|x-y|$. Find $\alpha>0$ such that $|f(x)-f(y)|>M(|x-y|+\alpha)$. By continuity of $f$ at $x$, there is a $\delta>0$ such that 
$$|f(x)-f(z)|<|f(x)-f(y)|-M(|x-y|+\alpha)$$
whenever $|x-z|<\delta$. Then, for each $z\in[0,1]$ such that $|x-z|<\min\{\delta,\alpha\}$, we have:
$$|f(z)-f(y)|\geq||f(x)-f(y)|-|f(x)-f(z)||>M(|x-y|+\alpha)>M|y-z|.$$ 
Hence, $z\notin L_M(f)$. 
\end{proof}

This result implies that $L(f)$ is Borel (of type $\mathtt{F_\sigma}$). Thus, we can consider Lebesgue's measure of the sets $L_M(f)$ and $L(f)$.

\begin{lem}
\label{lem1}
For each $M\in\mathbb{N}$ and $c\in (0,1]$, the set of functions $f\in C[0,1]$ such that $\lambda(L_M(f))\geq c$ is closed in $C[0,1]$.
\end{lem}

\begin{proof}
Fix a sequence $(f_n)\subseteq C[0,1]$ converging to some $f\in C[0,1]$ and such that $\lambda(L_M(f_n))\geq c$ for each $n$. We need to show that $\lambda(L_M(f))\geq c$. Suppose, to the contrary, that $\lambda(L_M(f))<c$. Using regularity of Lebesgue's measure, find an open set $G\subseteq[0,1]$ such that $L_M(f)\subseteq G$ and $\lambda(G)<c$.

For each $n$, there is an $x_n\in L_M(f_n)\setminus G$ (as $\lambda(G)<\lambda(L_M(f_n))$). Since $[0,1]$ is compact, without loss of generality we may assume that $(x_n)$ converges to some $x\in[0,1]$. Observe that $x\notin L_M(f)$ as $G$ is an open hull of $L_M(f)$ and whole sequence $(x_n)$ is outside of $G$. However, 
$$|f(x)-f(y)|\leq |f(x)-f_n(x)|+|f_n(x)-f_n(x_n)|+|f_n(x_n)-f_n(y)|+|f_n(y)-f(y)|\leq$$
$$\left\|f-f_n\right\|+|f_n(x)-f_n(x_n)|+M|x_n-y|+\left\|f-f_n\right\|$$
for each $y\in[0,1]$. Convergence of $(f_n)$ to $f$ implies equicontinuity of $(f_n)$ at $x$. So, if $n$ tends to infinity, we get that $|f(x)-f(y)|\leq M|x-y|$. This contradicts $x\notin L_M(f)$. 
\end{proof}

First, we want to focus on functions differentiable almost everywhere. 

\begin{prop}
\label{ae}
The set of functions differentiable almost everywhere is Haar-$1$ in $C[0,1]$.
\end{prop}

\begin{proof}
Let $\mathcal{A}$ denote the set of functions $f\in C[0,1]$ such that $\lambda(L(f))=1$. Note that each function differentiable almost everywhere is in $\mathcal{A}$. Thus, if we show that $\mathcal{A}$ is Haar-$1$ in $C[0,1]$, the thesis will follow. 

Recall that $L(f)=\bigcup_{M\in\mathbb{N}}L_M(f)$ and $L_M(f)\subseteq L_{M+1}(f)$. Thus, $\lambda(L(f))=\lim_{M\to\infty}\lambda(L_M(f))$. Consequently, 
$$\mathcal{A}=\bigcap_{k\in\mathbb{N}}\bigcup_{M\in\mathbb{N}}\left\{f\in C[0,1]:\ \lambda(L_M(f))\geq\frac{k-1}{k}\right\}$$
and $\mathcal{A}$ is Borel by Lemma \ref{lem1}.

Observe that $\mathcal{A}-\mathcal{A}$ is meager as a subset of the set of somewhere Lipschitz functions, which is meager by the Banach theorem \cite{Banach} (see also the proof of Proposition \ref{Banach2dim}). Indeed, if $f\in\mathcal{A}-\mathcal{A}$ then there are $g,h\in\mathcal{A}$ such that $f=g-h$. Since $\lambda(L(g))=1=\lambda(L(h))$, there is $x\in\lambda(L(g))\cap\lambda(L(h))$. Thus, there are $M_1,M_2\in\mathbb{N}$ such that $|g(x)-g(y)|\leq M_1|x-y|$ and $|h(x)-h(y)|\leq M_2|x-y|$ for all $y\in[0,1]$. Thus, for each $y\in[0,1]$ we have:
\[|f(x)-f(y)|\leq|g(x)-g(y)|+|h(x)-h(y)|\leq (M_1+M_2)|x-y|.\]

The thesis follows now from \cite[Corollary 16.3]{1}, which states that a Borel set $B$ is Haar-$1$ if and only if $B-B$ is meager.
\end{proof}

Actually, this reasoning gives us something more.

\begin{cor}
The set of functions $f\in C[0,1]$ such that $\lambda(D(f))>\frac{1}{2}$ is Haar-$1$ in $C[0,1]$.
\end{cor}

\begin{proof}
Let $\mathcal{A}$ denote the set of functions $f\in C[0,1]$ such that $\lambda(L(f))>\frac{1}{2}$. Similarly as in the proof of Proposition \ref{ae}, one can show that $\mathcal{A}$ is Borel. If we show that $\mathcal{A}$ is Haar-$1$ in $C[0,1]$, the thesis will follow. 

We will show that $\mathcal{A}-\mathcal{A}$ is a subset of the set of somewhere Lipschitz functions. Indeed, if $f\in\mathcal{A}-\mathcal{A}$ then there are $g,h\in\mathcal{A}$ such that $f=g-h$. Since $\lambda(L(g))>\frac{1}{2}$ and $\lambda(L(h))>\frac{1}{2}$, there is $x\in\lambda(L(g))\cap\lambda(L(h))$. Proceeding in the same way as in the proof of Proposition \ref{ae}, we can conclude that $f$ is Lipschitz in $x$.

From now on the proof is entirely the same as the proof of Proposition \ref{ae}.
\end{proof}

Now, we will study functions $f\in C[0,1]$ such that $\lambda(D(f))\in(0,1)$.

\begin{prop}
\label{Positive1}
Let $\I$ be a $\sigma$-ideal on $[0,1]$ containing no interval. The set of functions $f\in C[0,1]$ such that neither $D(f)$ nor $[0,1]\setminus D(f)$ belongs to $\I$ is not Haar-finite in $C[0,1]$.
\end{prop}

\begin{proof}
Denote by $\mathcal{A}$ the set of functions $f\in C[0,1]$ such that $D(f)\notin\I$ and $[0,1]\setminus D(f)\notin\I$.

Let $\phi\colon \{0,1\}^\N\to C[0,1]$ be a continuous map. We need to show that $\phi^{-1}[\mathcal{A}+h]$ is infinite for some $h\in C[0,1]$.

If $\phi[\{0,1\}^\N]$ is finite, then let $g$ be any element of $\mathcal{A}$ and $h\in C[0,1]$ be such that $\phi^{-1}[\{h\}]$ is infinite. Observe that $h\in\mathcal{A}+h-g$. Hence, $\phi^{-1}[\mathcal{A}+h-g]$ is infinite as well. 

Assume now that $\phi[\{0,1\}^\N]$ is infinite and take any injective convergent sequence $(f_n)\in\phi[\{0,1\}^\N]$. Denote $f=\lim_{n}f_n$. For each $n$, let $a_n,b_n,c_n\in(\frac{1}{n+2},\frac{1}{n+1})$ be such that $a_n<b_n<c_n$ and denote $I_n^1=[a_n,b_n]$ and $I_n^2=[b_n,c_n]$. For all $n$ let also $J_n=[c_{n+1},a_n]$ and $g_n\colon J_n\to\R$ be the linear function given by $g_n(c_{n+1})=0$ and $g_n(a_{n})=(f_n-f_{n+1})(a_n)$. Fix any nowhere differentiable function $z\in C[0,1]$ such that $\left\|z\right\|\leq 1$ and $z(b_n)=z(c_n)=0$ for each $n$. Such function exists as given any nowhere differentiable function $\tilde{z}\in C[0,1]$ with $z(0)=z(1)=0$ (for instance, the Takagi function) we can define $z\in C[0,1]$ in such a way that $z\upharpoonright[b_n,c_n]=\frac{\tilde{z}\circ r_n}{\left\|z\right\|}$ and $z\upharpoonright[c_{n+1},b_n]=\frac{\tilde{z}\circ s_n}{\left\|z\right\|}$ for each $n$, where $r_n:[b_n,c_n]\to[0,1]$ and $s_n:[c_{n+1},b_n]\to[0,1]$ are linear bijections. Define $h\colon[0,1]\to\mathbb{R}$ by:
\begin{itemize}
    \item $h\upharpoonright [c_1,1]=f_1(c_1)$;
    \item $h\upharpoonright I_n^1=f_n\upharpoonright I_n^1$ for each $n\in\N$;
    \item $h\upharpoonright I_n^2=f_n-\frac{z}{n}\upharpoonright I_n^2$ for each $n\in\N$;
    \item $h\upharpoonright J_n=(f_{n+1}+g_n)\upharpoonright J_n$ for each $n\in\N$;
    \item $h(0)=f(0)$.
\end{itemize}
Note that the function $h$ is continuous (continuity in $0$ follows from $\lim_{n}f_n=f$, $\lim_n \sup_{x\in J_n}g_n(x)=0$, and $\lim_{n}\frac{\left\|z\right\|}{n}=0$) and $f_n\in\mathcal{A}+h$ for each $n\in\N$ (as $(f_n-h)\upharpoonright I_n^1$ is constant and $(f_n-h)\upharpoonright I_n^2$ is nowhere differentiable). Since $(f_n)$ is injective, we conclude that $\phi^{-1}[\mathcal{A}+h]$ is infinite.
\end{proof}

The next Corollary is immediate.

\begin{cor}
The set of functions $f\in C[0,1]$ such that $\lambda(D(f))\in(0,1)$ is not Haar-finite in $C[0,1]$.
\end{cor}

Now, we will show that the set of functions $f\in C[0,1]$ such that $D(f)$ is of positive Lebesgue's measure is Haar-countable. The next lemma is a straightforward modification of one of the proofs showing that the Weierstrass function is nowhere Lipschitz.

\begin{lem}
\label{lemWeierstrass}
There is an $a\in(0,1)$ and there is a $b\in\mathbb{N}$ such that for any $x\in(0,1]$ and any increasing sequence $(s_j)$ of positive integers, and any $(r_j)\in\{-1,1\}^\mathbb{N}$, the expression
$$\left|\frac{\sum_{j=0}^{\infty}r_j a^{s_j}\left(\cos\left(b^{s_j} \pi y\right)-\cos\left(b^{s_j} \pi x\right)\right)}{y-x}\right|,$$
where $y$ runs over $[0,1]$, is unbounded.
\end{lem}

\begin{proof}
Let $a\in(0,1)$ and $b\in\mathbb{N}$ be such that $b$ is odd, $ab>1$, and $\frac{2}{3}-\frac{4a}{1-a}-\frac{\pi}{ab-1}>0$ (for instance, any $a<\frac{1}{13}$ and any odd $b\in\mathbb{N}$ with $b>\frac{3\pi+1}{a}$ are good as in this case we have $\frac{4a}{1-a}<\frac{1}{3}$ and $\frac{\pi}{ab-1}<\frac{1}{3}$).

For each $m\in\mathbb{N}$, let $w_m\in\mathbb{Z}$ be such that $x_m=x b^{s_m}-w_m\in(-\frac{1}{2},\frac{1}{2}]$ and define $y_m=(w_m-1)/b^{s_m}$. As $-\frac{1}{2}<x b^{s_m}-w_m\leq\frac{1}{2}$, we get $xb^{s_m}-\frac{3}{2}\leq w_m-1<xb^{s_m}-\frac{1}{2}$ and $x-\frac{3}{2b^{s_m}}\leq y_m<x-\frac{1}{2b^{s_m}}$. Thus, $y_m\in[0,1]$ for sufficiently large $m$ (as $x>0$).

Fix any $m\in\mathbb{N}$. Since $|\sin(z)/z|\leq 1$, using the formula for the difference of cosines, for each $z$ we have:
$$\left|\frac{\sum_{j=0}^{m-1}r_j a^{s_j}\left(\cos\left(b^{s_j} \pi y_m\right)-\cos\left(b^{s_j} \pi x\right)\right)}{y_m-x}\right|=$$
$$=\left|\sum_{j=0}^{m-1}r_j (ab)^{s_j} \pi \sin\left(\frac{b^{s_j} \pi (y_m+x)}{2}\right)\frac{\sin\left(\frac{b^{s_j} \pi (y_m-x)}{2}\right)}{\frac{b^{s_j} \pi (y_m-x)}{2}}\right|\leq$$
$$\leq \pi\sum_{j=0}^{s_m-1} (ab)^{j}=\frac{\pi((ab)^{s_m}-1)}{ab-1}\leq\frac{\pi(ab)^{s_m}}{ab-1}.$$

We will need two observations: if $j\geq m$, then $\cos(b^{s_j}\pi y_m)=\cos(b^{s_j-s_m}\pi (w_m-1))=(-1)^{w_m-1}=-(-1)^{w_m}$ (as $b$ is odd and $w_m\in\mathbb{Z}$) and 
$$\cos(b^{s_j}\pi x)=\cos(b^{s_j-s_m}\pi w_m)\cos(b^{s_j-s_m}\pi x_m)-\sin(b^{s_j-s_m}\pi w_m)\sin(b^{s_j-s_m}\pi x_m)=$$
$$(-1)^{w_m}\cos(b^{s_j-s_m}\pi x_m)$$
(as $x=(x_m+w_m)/b^{s_m}$ and $\sin(b^{s_j-s_m}\pi w_m)=0$ by $b,w_m\in\mathbb{Z}$).

Thus, as $b^{s_m}(x-y_m)=1+x_m$ and $m$ is fixed, we have:
$$\left|\frac{\sum_{j=m}^{\infty}r_j a^{s_j}\left(\cos\left(b^{s_j} \pi y_m\right)-\cos\left(b^{s_j} \pi x\right)\right)}{y_m-x}\right|=$$
$$=\left|(ab)^{s_m}\frac{\sum_{j=m}^{\infty}r_j a^{s_j-s_m}(-1)(-1)^{w_m}(1+\cos(b^{s_j-s_m}\pi x_m))}{-(1+x_m)}\right|=$$
$$=\left|(ab)^{s_m}\sum_{j=m}^{\infty}r_j a^{s_j-s_m}\frac{1+\cos(b^{s_j-s_m}\pi x_m)}{1+x_m}\right|.$$
Recall that $x_m\in(-\frac{1}{2},\frac{1}{2}]$, which implies $\cos(\pi x_m)\geq 0$. Hence, we can bound the above from below by:
$$ (ab)^{s_m}\left(\frac{1}{1+\frac{1}{2}}-\sum_{j=m+1}^{\infty}a^{s_j-s_m}\frac{1+\cos(b^{s_j-s_m}\pi x_m)}{1+x_m}\right)\geq$$
$$\geq (ab)^{s_m}\left(\frac{2}{3}-4\sum_{j=1}^{\infty}a^{j}\right)=(ab)^{s_m}\left(\frac{2}{3}-\frac{4a}{1-a}\right).$$

Therefore, we get:
$$\left|\frac{\sum_{j=0}^{\infty}r_j a^{s_j}\left(\cos\left(b^{s_j} \pi y_m\right)-\cos\left(b^{s_j} \pi x\right)\right)}{y_m-x}\right|\geq$$

$$\geq\left|\left|\frac{\sum_{j=m}^{\infty}r_j a^{s_j}\left(\cos\left(b^{s_j} \pi y_m\right)-\cos\left(b^{s_j} \pi x\right)\right)}{y_m-x}\right|-\right.$$
$$-\left.\left|\frac{\sum_{j=0}^{m-1}r_j a^{s_j}\left(\cos\left(b^{s_j} \pi y_m\right)-\cos\left(b^{s_j} \pi x\right)\right)}{y_m-x}\right|\right|\geq$$
$$\geq (ab)^{s_m}\left(\frac{2}{3}-\frac{4a}{1-a}-\frac{\pi}{ab-1}\right)\xrightarrow{m\to\infty}\infty,$$
since $ab>1$ and $\frac{2}{3}-\frac{4a}{1-a}-\frac{\pi}{ab-1}>0$.
\end{proof}

Recall that the $\sigma$-ideal of Lebesgue's null sets is ccc, i.e., every family of pairwise disjoint Borel sets of positive Lebesgue's measure is countable.

\begin{prop}
\label{Positive2}
The set of functions $f\in C[0,1]$ such that $\lambda(D(f))>0$ is Haar-countable in $C[0,1]$.
\end{prop}

\begin{proof}
Denote by $\mathcal{A}$ the set of functions $f\in C[0,1]$ such that $\lambda(L(f))>0$ and note that each function $f\in C[0,1]$ such that $\lambda(D(f))>0$ is in $\mathcal{A}$. Moreover, analogously to the proof of Proposition \ref{ae}, we get that $\mathcal{A}$ is Borel, because $\lambda(L(f))=\lim_{M\to\infty}\lambda(L_M(f))$ and 
$$\mathcal{A}=\bigcup_{k\in\mathbb{N}}\bigcup_{M\in\mathbb{N}}\left\{f\in C[0,1]:\ \lambda(L_M(f))\geq\frac{1}{k}\right\}.$$

Let $a\in(0,1)$ and $b\in\mathbb{N}$ be as in Lemma \ref{lemWeierstrass} and let $\phi\colon \{0,1\}^\N\to C[0,1]$ be given by: 
$$\phi(\alpha)(x)=\sum_{j=0}^{\infty}(-1)^{\alpha(j)}a^j\cos(b^j \pi x)$$ 
for each $\alpha\in \{0,1\}^\N$ and $x\in[0,1]$. Continuity of $\phi$ is obvious as $\lim_m\sum_{j=m}^{\infty}a_j=0$ and, consequently,
\[\lim_m\left(\sum_{j=m}^{\infty}(-1)^{\alpha(j)}a^j\cos(b^j \pi x)\right)=0.\]
Continuity of each $\phi(\alpha)$ follows from Weierstrass M-test and uniform limit theorem.

We claim that $\phi$ witnesses that $\mathcal{A}$ is Haar-countable. Suppose, to the contrary, that there is an $h\in C[0,1]$ and an uncountable set $T \subseteq \{0,1\}^\N$ such that $\phi(\alpha)\in\mathcal{A}+h$ for each $\alpha\in T$. Without loss of generality, we may assume that $\{i\in\N:\ \alpha(i)\neq \beta(i)\}$ is infinite for each two distinct $\alpha,\beta\in T$ (since for every $\alpha\in \{0,1\}^\N$ there are only countably many $\beta\in \{0,1\}^\N$ such that $\{i\in\N:\ \alpha(i)\neq \beta(i)\}$ is finite).

For $\alpha\in T$, consider the sets $L(\phi(\alpha)-h)\setminus\{0\}$. They are Borel (cf. Lemma \ref{lem3}) and of positive Lebesgue's measure. By the ccc property, we conclude that there are $\alpha,\beta\in T$, $\alpha\neq\beta$ such that $L(\phi(\alpha)-h)\cap L(\phi(\beta)-h)\setminus\{0\}\neq\emptyset$. Let $x$ belong to that set (note that $x\in(0,1]$).

Let $(s_i)$ be the increasing enumeration of the set $\{i\in\N:\ \alpha(i)\neq \beta(i)\}$ and denote $r_j=\beta(s_j)-\alpha(s_j)\in\{-1,1\}$. Observe that:
$$\frac{\phi(\alpha)(y)-\phi(\beta)(y)-\phi(\alpha)(x)-\phi(\beta)(x)}{y-x}=2\frac{\sum_{j=0}^{\infty}r_j a^{s_j}\left(\cos\left(b^{s_j} \pi y\right)-\cos\left(b^{s_j} \pi x\right)\right)}{y-x}$$
for any $y\in[0,1]$. Therefore, by Lemma \ref{lemWeierstrass}, the above expression is unbounded. On the other hand, there are $M_\alpha,M_\beta\in\mathbb{N}$ witnessing that $\phi(\alpha)-h$ and $\phi(\beta)-h$ are Lipschitz at $x$. Thus:
$$
\left|\frac{\phi(\alpha)(y)-\phi(\beta)(y)-\phi(\alpha)(x)-\phi(\beta)(x)}{y-x}\right|\leq \left|\frac{\phi(\alpha)(y)-h(y)-\phi(\alpha)(x)-h(x)}{y-x}\right|+
$$
$$
+\left|\frac{\phi(\beta)(y)-h(y)-\phi(\beta)(x)-h(x)}{y-x}\right|\leq M_\alpha+M_\beta
$$
for any $y\in[0,1]$. This is a contradiction.
\end{proof}

The next table summarizes results of Sections $2$ and $3$.

\begin{tabular}{c|c|c|c}
$A=$ & $\{0\}$ & $(0,1)$ & $\{1\}$ \cr
\hline
$\{f\in \mathcal{SD}:\ \lambda(D(f))\in A\}\in$& $\text{Haar-}\mathcal{E}$ & $\text{Haar-countable}$ & Haar-$1$\\
\hline
$\{f\in \mathcal{SD}:\ \lambda(D(f))\in A\}\notin$& $\text{Haar-countable}$ & $\text{Haar-finite}$ & --\\
\end{tabular}

\section{Multidimensional case}

In this Section, we study nowhere differentiable functions on $[0,1]^k$, i.e., functions defined on $[0,1]^k$ which do not have a finite directional derivative at any point along any vector (at points from the boundary of $[0,1]^k$ we require that there is no finite directional one-sided derivative along any vector). Such functions exist (actually, in Proposition \ref{Banach2dim} we even show that the set of such functions is comeager in $C[0,1]^k$), however, it is hard to find a suitable example in the literature. Therefore, below we provide one for $k=2$ (with an informal proof).

\begin{ex}
Let $T \in C[0,1]$ be the Takagi function, i.e., 
$$T(x)=\sum_{n=0}^{+\infty}\frac{1}{2^n}\text{dist}(2^n x,\mathbb{Z}).$$
We define a new function on $[0,1]^2$:
\[F(x_1,x_2) = \tilde{T}\left(\sqrt{x_1^2+x_2^2}\right)\]
(here $\tilde{T}(x)=T(\text{frac}(x))$, where $\text{frac}(x)$ is the fractional part of $x$). Choose a point $(0,0)\neq (\chi_0, \chi_1) \in [0,1]^2$ and a unit vector $\upsilon\in\mathbb{R}^2$. Denote by $\ell$ a line that goes through $(\chi_1,\chi_2)$ and that is parallel to $\upsilon$. For every point $(x_1,x_2)$ of $\ell\cap[0,1]^2$, consider the circle $\mathcal{O}_{(x_1,x_2)}$ centered at the origin that passes through $(x_1,x_2)$. Let $x$ be the $x$-intercept of $\mathcal{O}_{(x_1,x_2)}$ and define a one-to-one correspondence between $[\inf\{\sqrt{x_1^2+x_2^2}:(x_1,x_2)\in\ell\cap[0,1]^2\},\sup\{\sqrt{x_1^2+x_2^2}:(x_1,x_2)\in\ell\cap[0,1]^2\}]$ and $\ell\cap [0,1]^2$ by $x \mapsto (x_1,x_2)$. We will always denote by $x$ a point on $[0,1]$ that corresponds to the point $(x_1,x_2)$ on $\ell$. Note that $F(x_1,x_2) = F(x,0) = \tilde{T}(x)$. Denote by $d_e^2$ the standard euclidean metric on $[0,1]^2$. Moreover, denote by $a$ and $b$ the slope of $\ell$ and its $y$-intercept respectively. Choose two sequences $(u^n)$ and $(v^n)$ that tend to $\chi$ (recall that $\chi$ is the $x$-intercept of $\mathcal{O}_{(\chi_1,\chi_2)}$). Now, for every $n$, assume that all elements of $(u^k)_{k \geq n}$ and $(v^k)_{k \geq n}$ are within some interval $(\alpha_n, \beta_n)$. Without loss of generality, we may assume that both $(\alpha_n)$ and $(\beta_n)$ tend to $\chi$. Using elementary analytical geometry and a little bit of estimation, we can show that:
\[\frac{\alpha_n \sqrt{a^2+1}}{\sqrt{\beta_n^2 (a^2+1) - b^2}} |u^n - v^n| \leq d_e^2\big((u_1^n, u_2^n), (v_1^n, v_2^n)\big) \leq \frac{\beta_n \sqrt{a^2+1}}{\sqrt{\alpha_n^2 (a^2+1) - b^2}} |u^n - v^n| \]
(i.e., the euclidean metric on $(\alpha_n, \beta_n)$ is equivalent to $d_e^2\upharpoonright \{x\in\ell:d_e^2((0,0),x)\in (\alpha_n, \beta_n)\}^2$). For simplicity, let $c_n = \frac{\alpha_n \sqrt{a^2+1}}{\sqrt{\beta_n^2 (a^2+1) - b^2}}$ and $d_n = \frac{\beta_n \sqrt{a^2+1}}{\sqrt{\alpha_n^2 (a^2+1) - b^2}}$. Observe that $(c_n)$ and $(d_n)$ have the same limit. Note that:
 \[c_n \frac{|F(u_1^n, u_2^n) - F(v_1^n, v_2^n)|}{d_e^2\big((u_1^n, u_2^n), (u_1^n, u_2^n)\big)} \leq \frac{|\tilde{T}(u^n) - \tilde{T}(v^n)|}{|u^n - v^n|} \leq d_n \frac{|F(u_1^n, u_2^n) - F(v_1^n, v_2^n)|}{d_e^2\big((u_1^n, u_2^n), (u_1^n, u_2^n)\big)}.\]
Finally, it follows by the Squeeze Theorem that if $F$ is differentiable at $(\chi_1,\chi_2)$ in the direction $\upsilon$, then $T$ must be differentiable at $\chi$. This is impossible. Thus, $F$ is not differentiable at any point of $[0,1]^2\setminus\{(0,0)\}$ in any direction. To complete the proof, proceed analogously for $(\chi_0, \chi_1)=(0,0)$ and use the fact that $T$ does not possess a finite one-sided derivative at $0$.
\end{ex}

Now, we want to show that, unlike the one-dimensional case, the set of somewhere differentiable functions on $[0,1]^k$ is not Haar-null. Actually, this follows from a more general fact. We will need the following notion.

A subset $A$ of an abelian Polish group $X$ is called \emph{thick} if for any compact set $K\subseteq X$ there is an $x\in X$ such that $K+x\subseteq A$ (for more on thick sets see \cite[Section 7]{1}). 

\begin{rem}
The following are equivalent for any Borel set $A$:
\begin{itemize}
    \item[(a)] $A$ is thick;
    \item[(b)] $A$ is not Haar-$(\mathcal{P}(X)\setminus\{X\})$;
    \item[(c)] $A$ is not Haar-$\I$ for any proper semi-ideal $\I$.
\end{itemize}
Moreover, if $A$ is arbitrary, then (b) and (c) are equivalent and (a) implies both of them.
\end{rem}

\begin{proof}
Indeed, (b)$\Longleftrightarrow$(c) is trivial. As $f[\{0,1\}^\N]$ is compact for every continuous $f$, thickness of $A$ ensures us that $A$ is not Haar-$(\mathcal{P}(X)\setminus\{X\})$. Conversely, if $A$ is not thick, then the compact set $K$ witnessing it is a continuous image of $\{0,1\}^\N$. This continuous map witnesses that $A$ is Haar-$(\mathcal{P}(X)\setminus\{X\})$ provided that it is Borel.
\end{proof}

We are ready to prove the aforementioned general result.

\begin{prop}
Let $k>1$. The set of somewhere differentiable functions is thick in $C[0,1]^k$.
\end{prop}

\begin{proof}
Firstly, we will assume that $k=2$. 

Let $C$ be the ternary Cantor set (which is homeomorphic to $\{0,1\}^\N$) and $\phi\colon C\to C[0,1]^2$ be continuous. We will define a continuous function $f\colon[0,1]^2\to\R$ such that for all $c\in C$ the function $\phi(c)-f$ is differentiable along $(0,1)$ in each point of the form $(c,x)$, $x\in [0,1]$.

Let $F\colon C\times[0,1]\to\R$ be given by $F\upharpoonright \{c\}\times[0,1]=\phi(c)\upharpoonright \{c\}\times[0,1]$ for each $c\in C$. 

We need to show that $F$ is continuous. Fix any $(c_0,y_0)\in C\times[0,1]$ and $\varepsilon>0$. Since $\phi(c_0)$ is continuous at $(c_0,y_0)$, there is an open neighborhood $(c_0,y_0)\in U\subseteq [0,1]^2$ such that $|\phi(c_0)(x,y)-\phi(c_0)(c_0,y_0)|<\frac{\varepsilon}{2}$ for every $(x,y)\in U$. Moreover, since $\phi$ is continuous, there is an open neighborhood $c_0\in V\subseteq [0,1]$ such that $\left\|\phi(c)-\phi(c_0)\right\|<\frac{\varepsilon}{2}$ for each $c\in V\cap C$. Thus, we have:
$$|F(c,y)-F(c_0,y_0)|=|\phi(c)(c,y)-\phi(c_0)(c_0,y_0)|\leq $$
$$\leq|\phi(c)(c,y)-\phi(c_0)(c,y)|+|\phi(c_0)(c,y)-\phi(c_0)(c_0,y_0)|\leq$$
$$\leq\left\|\phi(c)-\phi(c_0)\right\|+|\phi(c_0)(c,y)-\phi(c_0)(c_0,y_0)|<\frac{\varepsilon}{2}+\frac{\varepsilon}{2}=\varepsilon$$
for all $(c,y)\in U\cap ((V\cap C)\times [0,1])$. Hence, $F$ is continuous.

Using Tietze extension theorem we get a continuous function $f\colon[0,1]^2\to\R$ such that $f\upharpoonright C\times[0,1]=F$. Then, $\phi(c)-f\upharpoonright\{c\}\times[0,1]$ is constantly equal to $0$ for each $c\in C$. Hence, $\phi(c)-f$ is somewhere differentiable and we are done.

If $k>2$ then define $F\colon C\times[0,1]\times\{0\}^{k-2}\to\R$ by $F\upharpoonright \{c\}\times[0,1]\times\{0\}^{k-2}=\phi(c)\upharpoonright \{c\}\times[0,1]\times\{0\}^{k-2}$ for each $c\in C$ and observe that $F$ is continuous for the same reason as above. Thus, using Tietze extension theorem we get a continuous function $f\colon[0,1]^k\to\R$ such that $\phi(c)-f$ is somewhere differentiable. 
\end{proof}

Although the set of somewhere differentiable functions on $[0,1]^k$ is not Haar-$\I$ for any semi-ideal $\I$, it is meager. In the one-dimensional case this was shown by Banach (see \cite{Banach}). The authors find it surprising (and remain skeptical) that no multidimensional case of this theorem can be found in the literature, as the following reasoning is only a slight modification of the one-dimensional proof.

\begin{prop}
\label{Banach2dim}
The set of nowhere differentiable functions on $[0,1]^k$ is comeager in $C[0,1]^k$.
\end{prop}

\begin{proof}
We will consider $k = 2$, but the argument works for other $k$'s as well. Let $E^{0,0}_n$ be the set of these functions $f \in C[0,1]^2$ such that there is a pair $(x,y)\in[0,1-\frac{1}{n+1}]^2$ and there is a unit vector $(v_1,v_2) \in [\frac{1}{n+1},1] \times [-1,1]$ such that for all $h \in (0,\frac{1}{n})$ it holds that $|f((x,y) + h(v_1,v_2))-f(x,y)| < nh$. Define $E^{0,1}_n$, $E^{1,0}_n$ and $E^{1,1}_n$ analogously replacing $R_n^{0,0}=[0,1-\frac{1}{n+1}]^2$ with $R_n^{0,1}=[0,1-\frac{1}{n+1}]\times[\frac{1}{n+1},1]$, $R_n^{1,0}=[\frac{1}{n+1},1]\times[0,1-\frac{1}{n+1}]$ and $R_n^{1,1}=[\frac{1}{n+1},1]^2$, respectively. Similarly, let $F^{i,j}_n$ be the set of these functions $f \in C[0,1]^2$ such that there is a pair $(x,y)\in R_n^{i,j}$ such that for all $h \in (0,\frac{1}{n})$ it holds that $|f(x,y+h)-f(x,y)| < nh$. It suffices to prove that for each $i,j\in\{0,1\}$ and each $n \in \mathbb{N}$:
\begin{itemize}
    \item[a)] $E_n^{i,j}$ is closed;
    \item[b)] $F_n^{i,j}$ is closed;
    \item[c)] $E_n^{i,j}$ is nowhere dense;
    \item[d)] $F_n^{i,j}$ is nowhere dense;
\end{itemize}
and also that:
\begin{itemize}
    \item[e)] $C[0,1]^2 \backslash \bigcup_{n\in\N} \bigcup_{i,j\in\{0,1\}}(E^{i,j}_n \cup F^{i,j}_n)$ is a subset of the set of nowhere differentiable functions on $[0,1]^2$.
\end{itemize}

For $a)$ and $b)$, it is enough to use the argument from the standard one-dimensional case. We will show that $E_m^{0,0}$ is closed. The other cases are similar.

Let $(f_n)\subseteq E_m^{0,0}$ and denote $f=\lim_n f_n$. For each $n\in\mathbb{N}$ there is $(x_n,y_n)\in[0,1-\frac{1}{m+1}]^2$ and there is a unit vector $(v^n_1,v^n_2) \in [\frac{1}{m+1},1] \times [-1,1]$ such that for all $h \in (0,\frac{1}{m})$ it holds that $|f_n((x_n,y_n) + h(v^n_1,v^n_2))-f_n(x_n,y_n)| < mh$. By the Bolzano-Weierstrass theorem, passing to a subsequence, if necessary, without loss of generality we can assume that $\lim_n(x_n,y_n)=(x,y)\in[0,1-\frac{1}{m+1}]^2$ and $\lim_n(v^n_1,v^n_2)=(v_1,v_2) \in [\frac{1}{m+1},1] \times [-1,1]$. Fix $h\in (0,\frac{1}{m})$ and observe that:
\[\left|f((x,y) + h(v_1,v_2))-f(x,y)\right|\leq
\left|f((x,y) + h(v_1,v_2))-f((x_n,y_n) + h(v^n_1,v^n_2))\right|+\]
\[+\left|f((x_n,y_n) + h(v^n_1,v^n_2))-f_n((x_n,y_n) + h(v^n_1,v^n_2))\right|+\]
\[+\left|f_n((x_n,y_n) + h(v^n_1,v^n_2))-f_n((x_n,y_n))\right|+
\left|f_n((x_n,y_n))-f((x_n,y_n))\right|+\]
\[+\left|f((x_n,y_n))-f((x,y))\right|\leq
\left|f((x,y) + h(v_1,v_2))-f((x_n,y_n) + h(v^n_1,v^n_2))\right|+\]
\[+\left\|f-f_n\right\|+
mh+
\left\|f-f_n\right\|+
\left|f((x_n,y_n))-f((x,y))\right|.\]
If $n\to+\infty$ then continuity of $f$ at $(x,y)$ and at $((x,y) + h(v_1,v_2))$ together with $f=\lim_n f_n$ gives us that:
\[\left|f((x,y) + h(v_1,v_2))-f(x,y)\right|\leq mh\]
and we can conclude that $f\in E_m^{0,0}$.

For $c)$, we will prove that for every two-dimensional piecewise linear function $g$ and every $\epsilon > 0$ there is a function $f \not\in E^{0,0}_n$ such that the norm of $f-g$ is below $\epsilon$ (the cases of $E^{0,1}_n$, $E^{1,0}_n$ and $E^{1,1}_n$ are analogous). As the set of all two-dimensional piecewise linear functions is dense in $C[0,1]^2$, the thesis will follow.

Let $g$ and $\epsilon$ be defined as above and let $M$ be equal to the maximal slope of $g$. Subsequently, choose $m$ such that $\frac{m}{n+1}\frac{\epsilon}{2} > M+n$. Now, define a function $f(x,y)=g(x,y) + \frac{\epsilon}{2} \text{dist}(mx,\mathbb{Z})$. It is easy to see that $|f(x,y)-g(x,y)| < \epsilon$ for all $(x,y)$. Let $(x,y) \in [0,1-\frac{1}{n+1}]^2$ and let $v=(v_1,v_2) \in [\frac{1}{n+1},1] \times [-1,1]$ be a unit vector. Observe that:
\[\left|\frac{\partial f}{\partial v}\right| = \left|\frac{\partial g}{\partial v} + \frac{\epsilon}{2}m v_1\frac{\partial \text{dist}(\cdot,\mathbb{Z})}{\partial x}\right| > n\]
as $\left|\frac{\partial \text{dist}(\cdot,\mathbb{Z})}{\partial x}\right|=1$ and $v_1>\frac{1}{n+1}$.

For $d)$, we just repeat an argument we just gave: we choose $m$ such that $m\frac{\epsilon}{2} > M+n$, define $f(x,y)=g(x,y) + \frac{\epsilon}{2} \text{dist}(my,\mathbb{Z})$ and observe that:
\[\left|\frac{\partial f}{\partial y}\right| = \left|\frac{\partial g}{\partial y} + \frac{\epsilon}{2}m\frac{\partial \text{dist}(\cdot,\mathbb{Z})}{\partial y}\right| > n.\]

Finally, for $e)$, let us suppose that $f \in C[0,1]^2 \backslash (\bigcup_{n \in \mathbb{N}}\bigcup_{i,j\in\{0,1\}} E^{i,j}_n \cup \bigcup_{n \in \mathbb{N}} F^{i,j}_n)$, choose $(x,y) \in (0,1)^2$, and choose a unit vector $v$. There are two cases to consider, but they essentially come down to the same argument. If the vector $v$ is different than $(0,1)$, we will use the sets $E^{i,j}_n$, if not, we will use $F^{i,j}_n$. Without loss of generality, we will continue under assumption that $v = (0,1)$. Let $i,j\in\{0,1\}$ be such that $(x,y)\in\bigcap_{n\in\mathbb{N}} R_n^{i,j}$. Since $f \in \bigcap_{n \in \mathbb{N}} C[0,1]^2 \backslash F^{i,j}_n$, it follows that for all $n \in \mathbb{N}$ there is an $h_n \in (0,\frac{1}{n})$ such that $|f(x,y+h_n) - f(x,y)| \geq n h_n$. It is now easy to see that $f$ is nowhere differentiable.
\end{proof}

We end with an open problem which occurred during our studies.

Here, we examine the function $F\colon[0,1]^2\to\R$ given by $F(x,y)=T(x)+T(y)$, where $T\colon[0,1]\to\R$ is the Takagi's function. We considered this function in the context of nowhere differentiable functions on $[0,1]^2$ as $F$ seemed to be a nice example of such a function. It is obvious that it is differentiable neither along $(1,0)$ nor along $(0,1)$. Moreover, we managed to obtain one partial result: $F$ is differentiable at no point along $(1,1)$. However, we were unable to solve the following. 

\begin{problem}
Is $F$ nowhere differentiable on $[0,1]^2$?
\end{problem}


\begin{thebibliography}{HD}

\normalsize

\bibitem{Allaart}
P.C. Allaart, K. Kawamura,
\emph{The improper infinite derivatives of Takagi's nowhere-differentiable function},
J. Math. Anal. Appl. 372 (2010), 656--665.

\bibitem{survey}
P.C. Allaart, K. Kawamura,
\emph{The Takagi function: a survey},
Real Anal. Exchange 37 (2011), 1--54.

\bibitem{Balcerzak}
M. Balcerzak, 
\emph{Can ideals without ccc be interesting?}, 
Top. App. 55 (1994), 251--260.

\bibitem{Banach}
S. Banach,
\emph{\"Uber die Baire'sche Kategorie gewisser Funktionenmengen},
Studia Math. 3 (1931), 174--179.

\bibitem{1}   
T. Banakh, S. G\l \c{a}b, E. Jab\l o\'{n}ska, J. Swaczyna,
\emph{Haar-$\mathcal{I}$ sets: looking at small sets in Polish groups through compact glasses},
submitted, available online at: https://arxiv.org/abs/1803.06712 (last accessed July 5th, 2020).

\bibitem{BLR}
T. Banakh, N. Lyaskovska, D. Repov\v s, 
\emph{Packing index of subsets in Polish groups}, 
Notre Dame J. Formal Logic 50 (2009), 453--468.

\bibitem{Judah}
T. Bartoszy\'{n}ski, H. Judah, 
\emph{Set theory: On the structure of the real line},
A K Peters, Ltd., Wellesley, MA, 1995.

\bibitem{Christensen}
J.P.R. Christensen, 
\emph{On sets of Haar measure zero in abelian Polish groups}, 
Israel J. Math. 13 (1972), 255--260.

\bibitem{Darji}
U.B. Darji, 
\emph{On Haar meager sets}, 
Top. Appl. 160 (2013), 2396--2400.

\bibitem{DK}
U. B. Darji, T. Keleti, 
\emph{Covering $\mathbb{R}$ with translates of a compact set},
Proc. Amer. Math. 131 (2003), 2593--2596.

\bibitem{Elekes}
M. Elekes, J. Stepr\=ans, 
\emph{Less than $2^\omega$ many translates of a compact nullsets may cover the real line}, 
Fund. Math. 181 (2004), 89--96.

\bibitem{Hunt}   
B.R. Hunt,
\emph{The prevalance of continuous nowhere differentiable functions}, 
Proc. Amer. Math. Soc. 3 (1994), 711--717.

\bibitem{HSY}   
B.R. Hunt, T. Sauer, J.A. Yorke,
\emph{Prevalence: a translation-invariant ''almost every'' on infinite-dimensional spaces}, 
Bull. Amer. Math. Soc. 27 (1992), 217--238.

\bibitem{ja}   
A. Kwela,
\emph{Haar-smallest sets},
Top. App. 267 (2019), 106892.

\bibitem{Takagi}   
T. Takagi,
\emph{A simple example of continuous function without derivative},
Proc. Phys. Math. Soc. Japan 1 (1903), 176--177.

\bibitem{Zahorski}   
Z. Zahorski,
\emph{Sur l'ensemble des points de non-d\'{e}rivabilit\'{e} d'une fonction continue},
Bulletin de la S. M. F. 74 (1946), 147--178.

\bibitem{Zakrzewski}   
P. Zakrzewski,
\emph{On Borel sets belonging to every invariant ccc $\sigma$-ideal on $2^\omega$},
Proc. Amer. Math. Soc. 141 (2013), 1055--1065.

\end{thebibliography}
\end{document}